\documentclass[a4paper,12pt]{article}
\usepackage{amssymb,amsmath,amsthm,latexsym}
\usepackage{amsfonts}
	\usepackage{amsfonts}
\usepackage{graphicx}
\usepackage[pdftex,bookmarks,colorlinks=false]{hyperref}
\usepackage{verbatim}
\usepackage{caption}
\usepackage{subcaption}
\usepackage[hmargin=1.2in,vmargin=1.2in]{geometry}

\newtheorem{theorem}{Theorem}[section]

\newtheorem{definition}[theorem]{Definition}

\newtheorem{lemma} [theorem]{Lemma}

\newtheorem{remark}[theorem]{Remark}

\title{\textbf {\sc A Study on the Sparing Number of the Corona of Certain Graphs}}

\author{{\bf K P Chithra} $^{{1},{\ast}}$, {\bf K A Germina} $^{2}$  and {\bf N K Sudev $^{3}$}
\\ \\
$^{1}${\small Naduvath Mana, Nandikkara P O} \\ {\small  Thrissur - 680301, Kerala, India.}\\ {\small email: {\em chithrasudev@gmail.com}}
\vspace{0.3cm}
\\
$^{2}${\small Department of Mathematics, School of Mathematical \& Physical Sciences} \\ {\small Central University of Kerala, Kasaragod - 671316, Kerala, India.}\\ {\small email: {\em srgerminaka@gmail.com}}
\vspace{0.3cm}
\\
$^{3}${\small Department of Mathematics, Vidya Academy of Science \& Technology} \\ {\small  Thalakkottukara, Thrissur - 680501, Kerala, India.}\\ {\small email: {\em sudevnk@gmail.com}}
\\ \vspace{0.3cm}
$^{\ast}$ {\small Corresponding author.}
}

\date{}
\begin{document}
\maketitle

\begin{abstract}
Let $\mathbb{N}_0$ be the set of all non-negative integers and $\mathcal{P}(\mathbb{N}_0)$ be its the power set. An integer additive set-indexer (IASI) is defined as an injective function $f:V(G)\to \mathcal{P}(\mathbb{N}_0)$ such that the induced function $f^+:E(G) \to \mathcal{P}(\mathbb{N}_0)$ defined by $f^+ (uv) = f(u)+ f(v)$ is also injective, where $f(u)+f(v)$ is the sum set of $f(u)$ and $f(v)$. If $f^+(uv)=k~\forall~uv\in E(G)$, then $f$ is said to be a $k$-uniform integer additive set-indexer. An integer additive set-indexer $f$ is said to be a weak integer additive set-indexer if $|f^+(uv)|=\max(|f(u)|,|f(v)|)~\forall ~ uv\in E(G)$. We have some  characteristics of the graphs which admit weak integer additive set-indexers. In this paper, we study about the sparing number of the corona of two graphs.
\end{abstract}


\section{Introduction}

For all  terms and definitions, not defined specifically in this paper, we refer to \cite{FH}, \cite{JAG} and \cite{HIS}. Unless mentioned otherwise, all graphs considered here are simple, finite and have no isolated vertices.

Let $\mathbb{N}_0$ denote the set of all non-negative integers. For all $A, B \subseteq \mathbb{N}_0$, the sum of these sets is denoted by  $A+B$ and is defined by $A + B = \{a+b: a \in A, b \in B\}$. The set $A+B$ is called the {\em sum set} of the sets $A$ and $B$. If either $A$ or $B$ is countably infinite, then their sum set is also countably infinite. Hence, the sets we consider here are all finite sets of non-negative integers. The cardinality of a set $A$ is denoted by $|A|$. The power set of  a set $A$ is denoted by $\mathcal{P}(A)$. 

We define an integral multiple of a set $A$, as the set, denoted by $n.A$,  every element of whose is an integral multiple of the corresponding element of $A$. Therefore, $|n.A|=|A|$.

\begin{definition}\label{D2}{\rm
\cite{GA} An {\em integer additive set-indexer} (IASI, in short) is defined as an injective function $f:V(G)\to \mathcal{P}(\mathbb{N}_0)$ such that the induced function $f^+:E(G) \to \mathcal{P}(\mathbb{N}_0)$ defined by $f^+ (uv) = f(u)+ f(v)$ is also injective}.
\end{definition}

\begin{lemma}
\cite{GS1} If $f$ is an integer additive set-indexer defined on a given graph $G$, then $\max(|f(u)|,|f(v)|) \le |f^+(uv)| \le |f(u)|\,|f(v)|, \forall ~ u,v\in  V(G)$.
\end{lemma}

\begin{definition}{\rm
\cite{GS1} An IASI $f$ is said to be a {\em weak IASI} if $|f^+(uv)|=\max(|f(u)|,|f(v)|)$ for all $u,v\in V(G)$. A weak IASI $f$ is said to be {\em weakly uniform IASI} if $|f^+(uv)|=k$, for all $u,v\in V(G)$ and for some positive integer $k$.  A graph which admits a weak IASI may be called a {\em weak IASI graph}.}
\end{definition}

The following is the necessary and sufficient condition for a given graph to admit a weak IASI. 

\begin{lemma}
\cite{GS1} A graph $G$ admits a weak integer additive set-indexer if and only if every edge of $G$ has at least one mono-indexed end vertex.
\end{lemma}

\begin{definition}{\rm
\cite{GS3} The cardinality of the labeling set of an element (vertex or edge) of a graph $G$ is called the {\em set-indexing number} of that element. An element (a vertex or an edge) of graph which has the set-indexing number $1$ is called a {\em mono-indexed element} of that graph.}
\end{definition}

\begin{definition}{\rm
\cite{GS3} The {\em sparing number} of a graph $G$ is defined to be the minimum number of mono-indexed edges required for $G$ to admit a weak IASI and is denoted by $\varphi(G)$.}
\end{definition}

\begin{theorem}\label{T-WSG}
\cite{GS3} A subgraph of weak IASI graph is also a weak IASI graph.
\end{theorem}

\begin{theorem}\label{T-WUC}
\cite{GS3} A graph $G$ admits a weak IASI if and only if $G$ is bipartite or it has at least one mono-indexed edge.
\end{theorem}

\begin{theorem}\label{T-SNBP}
\cite{GS3} The sparing number of a bipartite graph $G$ is $0$.
\end{theorem}

\begin{theorem}\label{T-WUOC}
\cite{GS3} An odd cycle $C_n$ has a weak IASI if and only if it has at least one mono-indexed edge. 
\end{theorem}

\begin{theorem}\label{T-NME}
\cite{GS3} Let $C_n$ be a cycle of length $n$ which admits a weak IASI, for a positive integer $n$. Then, $C_n$ has an odd number of mono-indexed edges when it is an odd cycle and has even number of mono-indexed edges, when it is an even cycle. 
\end{theorem}

\begin{theorem}\label{T-WUG}
\cite{GS4} The graph $G_1\cup G_2$ admits a weak IASI if and only if both $G_1$ and $G_2$ are weak IASI graphs. 
\end{theorem}

\begin{theorem}\label{T-WKN}
\cite{GS3} The sparing number of a complete graph $K_n$ is $\frac{1}{2}(n-1)(n-2)$.
\end{theorem}

The admissibility of weak IASI by certain graph products is established in \cite{GS7}.  In this paper, our discussion is about the sparing number of a particular product, called corona, of two weak IASI graphs. 
 
\section{Corona of Weak IASI Graphs}

\begin{definition}\label{D-5.1}{\rm
\cite{FFH} The {\em corona} of two graphs $G_1$ and $G_2$, denoted by $G_1\odot G_2$, is the graph obtained by taking one copy of $G_1$ (which has $n_1$ vertices) and $n_1$ copies of $G_2$ and then joining the $i$-th point of $G_1$ to every point in the $i$-th copy of $G_2$. The number of vertices and edges in $G_1\odot G_2$ are $n_1(1+n_2)$ and $m_1+n_1m_2+n_1n_2$ respectively, where $n_i$ and $m_i$ are respectively the order and size of the graph $G_i, i=1,2$.}
\end{definition}

In the corona $G_1\odot G_2$ of given graphs $G_1$ and $G_2$, we take $|V(G_2)|$ layers or copies of $G_2$ and to establish the adjacency between all vertices of each copy to the corresponding vertex of $G_1$. The weak IASIs of $G_1$ and $G_2$ may not induce a weak IASI for $G_1\odot G_2$. Hence, we have to define an IASI independently for a graph product. 

The following theorem establishes a necessary and sufficient condition for the corona of two  weak IASI graphs to admit a weak IASI.

\begin{theorem}\label{T-NMIEG}
\cite{GS7} Let $G_1$ and $G_2$ be two weak IASI graphs on $m$ and $n$ vertices respectively. Then, $G_1\odot G_2$ admits a weak IASI if and only if either $G_1$ is $1$-uniform or it has $r$ copies of $G_2$ that are 1-uniform, where $r$ is the number of vertices in $G_1$ that are not mono-indexed. 
 
\end{theorem}

In view of Theorem \ref{T-NMIEG}, we examine the sparing number of the corona of certain graph structures. First, we consider the corona of two paths. The following result determines the sparing number $P_m \odot P_n$.

\begin{theorem}\label{T-NMEPcP}
The sparing number of the corona $P_m\odot P_n$ is 
\begin{equation*}
\varphi(P_m\odot P_n)=
\begin{cases}
\frac{1}{4}(m+1)(3n+1) & ~~\text{if  both $m$ and $n$ are odd}\\ 
\frac{3}{4}(m+1)n & ~~\text{if  both $m$ is odd and $n$ is even}\\ 
\frac{1}{4}(3mn+2n+m) & ~~\text{if  both $m$ is even and $n$ is odd}\\ 
\frac{1}{4}n(3m+4) & ~~\text{if  both $m$ and $n$ are even}\\ 
\end{cases}
\end{equation*}
\end{theorem}
\begin{proof}
Let $P_m$ and $P_n$ be two paths of lengths $m$ and $n$ respectively. Note that $P_m$ has $m+1$ vertices and $P_n$ has $n+1$ vertices. Label the vertices of $P_m$ alternately by distinct singleton sets and distinct non-singleton sets. Then, the number of mono-indexed edges in $P_m$ is $0$. Now, the number of mono-indexed vertices in $P_m$ is $\frac{m+1}{2}$ if $m$ is odd and $\frac{m}{2}$ if $m$ is even. Hence, there are the following cases.

\noindent {\em Case 1:} Let $m$ be odd.  Then, exactly $\frac{m+1}{2}$ copies of $P_n$ must be $1$-uniform and the vertices of the remaining $\frac{m+1}{2}$ copies of $P_n$ can be labeled alternately by distinct singleton and distinct non-singleton set-labels. In this context, we have the following subcases.

\noindent {\em Subcase 1.1:} Let $n$ be an odd integer. Here, $\frac{m+1}{2}$ copies of $P_n$ are $1$-uniform. In the remaining $\frac{m+1}{2}$ copies, exactly $\frac{n+1}{2}$ vertices are mono-indexed. Then, $\frac{n+1}{2}$ edges between $P_m$ and such a copy of $P_n$ are mono-indexed. Therefore, the number of mono-indexed edges is $\frac{m+1}{2}n +\frac{(m+1)}{2}.\frac{(n+1)}{2} = \frac{1}{4}(m+1)(3n+1)$.

\noindent {\em Subcase 1.2:} Let $n$ be an even integer. Here, $\frac{m+1}{2}$ copies of $P_n$ are $1$-uniform and in the remaining $\frac{m+1}{2}$ copies, $\frac{n}{2}$ vertices are mono-indexed. Therefore, $\frac{n}{2}$ edges between $P_m$ and such a copy of $P_n$ are  mono-indexed. Hence, the number of mono-indexed edges is $\frac{m+1}{2}n+\frac{m+1}{2}.\frac{n}{2} = \frac{3}{4}(m+1)n$.

\noindent {\em Case 2:} Let $m$ be even.  Then, exactly $\frac{m}{2}$ copies of $P_n$ must be $1$-uniform and the vertices of the remaining $\frac{m+2}{2}$ copies of $P_n$ can be labeled alternately by distinct singleton and distinct non-singleton set-labels. Here, we have the following subcases.

\noindent {\em Subcase 2.1:} Let $n$ be an odd integer. Here, $\frac{m}{2}$ copies of $P_n$ are $1$-uniform. In the remaining $\frac{m+2}{2}$ copies, exactly $\frac{n+1}{2}$ vertices are mono-indexed. Then, $\frac{n+1}{2}$ edges between $P_m$ and such a copy of $P_n$ are mono-indexed. Therefore, the number of mono-indexed edges is $\frac{m}{2}n +\frac{(m+2)}{2}.\frac{(n+1)}{2} = \frac{1}{4}(3mn+2n+m)$.

\noindent {\em Subcase 2.2:} Let $n$ be an even integer. Here, $\frac{m+2}{2}$ copies of $P_n$ are $1$-uniform and in the remaining $\frac{m}{2}$ copies, $\frac{n}{2}$ vertices are mono-indexed. Therefore, $\frac{n}{2}$ edges between $P_m$ and such a copy of $P_n$ are  mono-indexed. Therefore, the number of mono-indexed edges is $\frac{m+2}{2}n+\frac{m}{2}.\frac{n}{2} = \frac{1}{4}n(3m+4)$.
\end{proof}

Next, we discuss the corona of two graphs in which one is a cycle. First, recall the following remark.

\begin{remark}\label{R-NMEC}{\rm
\cite{GS3} Due to Theorem \ref{T-WUOC}, we observe that the number of mono-indexed vertices in the cycle $C_n$ of length $n$ is at least $\frac{(n+1)}{2}$ if $n$ is odd and is at least $\frac{n}{2}$, if $n$ is even.}
\end{remark}

By Theorem \ref{T-NMIEG}, The corona of a cycle $C_m$ and a path $P_n$, denoted by $C_m\odot P_n$, admits a weak IASI if and only if either $C_m$ is $1$-uniform or $r$ copies of $P_n$ that are $1$-uniform, where $r$ is the number of vertices of $C_m$ that are not mono-indexed and $P_n\odot C_m$ admits a weak IASI if and only if either $P_n$ is $1$-uniform or $r$ copies of $C_m$ are $1$-uniform, where $r$ is the number of vertices of $P_n$ that are not mono-indexed. Hence, we have

\begin{theorem}
The sparing number of $C_m \odot P_n$ is
\begin{equation*}
\varphi(C_m\odot P_n)=
\begin{cases}
\frac{3}{4}mn & ~~\text{if  both $m$ and $n$ are odd}\\ 
\frac{1}{4}m(3n+1) & ~~\text{if  both $m$ is odd and $n$ is even}\\ 
1+\frac{1}{4}m(3m-1) & ~~\text{if  both $m$ is even and $n$ is odd}\\ 
\frac{1}{4}[n(3m+1)+(m+5)] & ~~\text{if  both $m$ and $n$ are even}\\ 
\end{cases}
\end{equation*} 
\end{theorem}
\begin{proof}
Consider a cycle $C_m$ and a path $P_n$. Note that the path $P_n$ has $m+n$ vertices, while $C_m$ has $m$ vertices. Here, we consider the following cases.

\noindent {\em Case-1:} Let $m$ is even. Then, exactly $\frac{m}{2}$ vertices of $C_m$ have non-singleton set-label. Therefore, by Theorem \ref{T-NMIEG}, $C_m\odot P_n$ contains $\frac{m}{2}$ copies of $P_n$ that are $1$-uniform. The remaining $\frac{m}{2}$ copies of $P_n$ can be labeled alternately by distinct singleton and non-singleton sets. For these $\frac{m}{2}$ copies, we have the following subcases.

\noindent {\em Subcase-1.1:} Let $n$ be an even integer. Then, $\frac{n+2}{2}$ vertices can be labeled by non-singleton set-labels and the remaining $\frac{n}{2}$ vertices have singleton set-labels. Therefore, there are $\frac{n}{2}$ mono-indexed edges between $C_m$ and these copies of $P_n$. Hence, the total number of mono-indexed edges in $C_m\odot P_n$ is $\frac{m}{2}n+\frac{m}{2}\frac{n}{2}= \frac{3}{4}mn$.

\noindent {\em Subcase-1.2:} Let $n$ be an odd integer. Then, for each copy of $P_n$ corresponding to a mono-indexed vertex of $C_m$, there are $\frac{n+1}{2}$ vertices that can be labeled by non-singleton sets and $\frac{n+1}{2}$ vertices can be labeled by singleton sets. Therefore, there are $\frac{n+1}{2}$ mono-indexed edges between $C_m$ and these copies of $P_n$. Hence, the total number of mono-indexed edges in $C_m\odot P_n$ is 
$\frac{m}{2}n+\frac{m}{2}.\frac{n+1}{2}= \frac{1}{4}m(3n+1)$.

\noindent {\em Case-2:} Let $m$ is odd. Then, by Theorem \ref{T-WUOC}, the cycle $C_m$ has at least one mono-indexed edge. That is, $C_m$ has $\frac{m-1}{2}$ vertices that are not mono-indexed and $\frac{m+1}{2}$ mono-indexed vertices. Then, $\frac{m-1}{2}$ copies of $P_n$ are $1$-uniform and the vertices of the corresponding copies of remaining $\frac{m+1}{2}$ vertices of $C_m$, can be labeled alternately by singleton sets and non-singleton sets. In this case also, we have two options.

\noindent {\em Subcase-2.1:} Let $n$ be an even integer. Then, as explained in the above case, $\frac{n}{2}$ vertices have singleton set-labels and the remaining $\frac{n+2}{2}$ vertices can be labeled by non-singleton set-labels. Therefore, there are $\frac{n}{2}$ mono-indexed edges between $C_m$ and these copies of $P_n$. Hence, the total number of mono-indexed edges in $C_m \odot P_n$ is $1+\frac{m-1}{2}n+\frac{m+1}{2}.\frac{n}{2}=1+\frac{1}{4}n(3m-1)$.

\noindent {\em Subcase-2.2:} Let $n$ be an odd integer. Then, for each copy of $P_n$ corresponding to a mono-indexed vertex of $C_m$, there are $\frac{n+1}{2}$ vertices that can be mono-indexed and the remaining $\frac{n+1}{2}$ vertices can be labeled by non-singleton sets. Therefore, there are $\frac{n+1}{2}$ mono-indexed edges between $C_m$ and these copies of $P_n$. Hence, the total number of mono-indexed edges in $C_m\odot P_n$ is $1+\frac{m-1}{2}n+\frac{m+1}{2}.\frac{n+1}{2} =\frac{1}{4}[n(3m+1)+(m+5)]$
\end{proof}

Since $C_m\odot P_n$ and $P_n\odot C_m$ are non-isomorphic graphs, the sparing numbers of these graphs are also different. Therefore, we prove the following theorem for the corona $P_n\odot C_m$.

\begin{theorem}
The sparing number of $P_n\odot C_m$ is
\begin{equation*}
\varphi(P_n\odot C_m)=
\begin{cases}
\frac{3}{4}(m+1)(n+1) & ~~\text{if  both $m$ and $n$ are odd}\\ 
\frac{3}{4}m(n+1) & ~~\text{if  both $m$ is even and $n$ is odd}\\ 
\frac{1}{4}[3n(m+1)+2(m+3)] & ~~\text{if  both $m$ is odd and $n$ is even}\\ 
\frac{1}{4}m(3n+2) & ~~\text{if  both $m$ and $n$ are even}. 
\end{cases}
\end{equation*} 
\end{theorem}
\begin{proof}
We consider the following cases.

\noindent {\em Case-1:} Let $n$ is odd. then, $\frac{n+1}{2}$ vertices can be labeled by non-singleton sets and $\frac{n+1}{2}$ vertices can be labeled by singleton sets. Therefore, by Theorem \ref{T-NMIEG}, $\frac{n+1}{{2}}$ copies of $C_m$ must be $1$-uniform. For the remaining $\frac{n+1}{{2}}$ copies of $C_m$, the vertices can be labeled alternately by singleton sets and non-singleton sets respectively. In this context, we have the following subcases.

\noindent {\em Subcase-1.1:} Let $m$ be an odd integer. Then, each copy $C_m$, corresponding to the mono-indexed vertices of $P_n$, can have at least one mono-indexed edge. In this case, it has $\frac{m+1}{2}$ mono-indexed vertices and $\frac{m-1}{2}$ vertices that are not mono-indexed. That is, there exist $\frac{m+1}{2}$ edges between $P_n$ and each of the relevant copies of $C_m$. Hence, the total number of mono-indexed edges in $P_n\odot C_m$ is $\frac{n+1}{2}m+\frac{n+1}{2}.\frac{m+1}{2}+\frac{n+1}{2}= \frac{3}{4}(m+1)(n+1)$.

\noindent {\em Subcase-1.2:} Let $m$ be an even integer. Then, $\frac{m}{2}$ vertices can be labeled by singleton set-labels and the remaining $\frac{m}{2}$ vertices have non-singleton set-labels. Then, there exist $\frac{m}{2}$ edges between $P_n$ and each of the relevant copies of $C_m$. Hence, the total number of mono-indexed edges in $P_n\odot C_m$ is $\frac{n+1}{2}m+\frac{n+1}{2}\frac{m}{2}= \frac{3}{4}m(+1)n$.

\noindent {\em Case-2:} Let $n$ is even. Then, $\frac{n+2}{2}$ vertices of $P_n$ have singleton set-label and $\frac{n}{2}$ vertices of $P_n$ have non-singleton set-label. Therefore, by Theorem \ref{T-NMIEG}, $\frac{n+2}{2}$ copies of $C_m$ are $1$-uniform in $C_m\odot P_n$ and the remaining $\frac{n}{2}$ copies of $C_m$ can be labeled alternately by distinct singleton and non-singleton sets. For these $\frac{m}{2}$ copies, we have the following subcases.

\noindent {\em Subcase-2.1:} Let $m$ be an odd integer. Then, for each copy of $C_m$ corresponding to a mono-indexed vertex of $P_n$, there are $\frac{m+1}{2}$ vertices that can be mono-indexed and the remaining $\frac{m-1}{2}$ vertices can be labeled by non-singleton sets. Then, there exist $\frac{m+1}{2}$ edges between $P_n$ and each of the relevant copies of $C_m$. Hence, the total number of mono-indexed edges in $P_n\odot C_m$ is $\frac{n}{2}m+\frac{n+2}{2}.\frac{m+1}{2}+\frac{n+2}{2}=\frac{1}{4}[3n(m+1)+2(m+3)]$.

\noindent {\em Subcase-2.2:} Let $m$ be an even integer. Then, $\frac{m}{2}$ vertices have singleton set-labels and the remaining $\frac{m}{2}$ vertices can be labeled by non-singleton set-labels. Then, there exist $\frac{m}{2}$ edges between $P_n$ and each of the relevant copies of $C_m$. Therefore, the total number of mono-indexed edges in $P_n \odot C_m$ is $\frac{n}{2}m+\frac{n+2}{2}.\frac{m}{2}=\frac{1}{4}m(3n+2)$.
\end{proof}

We now discuss on the the sparing number of the corona of two cycles.  Let $C_m$ and $C_n$ be two cycles that admit weak IASI. Then, by Theorem \ref{T-NMIEG}, the corona $C_m\odot C_n$ admits a weak IASI if and only if $C_m$ is $1$-uniform or $r$ copies of $C_m$ are $1$-uniform, where $r$ is the number of vertices in $C_m$ that are not mono-indexed. Hence, we have

\begin{theorem}\label{T-WIASI-GP-C5}
The sparing number of $C_m\odot C_n$ is
\begin{equation*}
\varphi(C_m\odot C_n)=
\begin{cases}
\frac{3}{4}mn & ~~\text{if ~$m$~ and ~$n$~ are even}\\ 
\frac{3}{4}m(n+1) & ~~\text{if ~$m$~ is even and ~$n$~ is odd}\\
1+\frac{1}{4}(3m-1)n & ~~\text{if ~$m$~ is odd and ~$n$~ is even}\\
2+\frac{1}{4}(3m-1)(n+1) & ~~\text{if ~$m$~ and ~$n$~ are odd}.
\end{cases}
\end{equation*}
\end{theorem}

\begin{proof}
Let $C_m$ and $C_n$ be two cycles which admit weak IASIs. Let $G=C_m\odot C_n$. We consider the problem in two cases. 

\noindent {\em Case-1:} Let $m$ be even. Then, $C_m$ has $\frac{m}{2}$ mono-indexed vertices and has $\frac{m}{2}$ vertices that are not mono-indexed. Then, by Theorem \ref{T-NMIEG}, $\frac{m}{2}$ copies of $C_n$ are $1$-uniform. Let $C_{n_i}$ be the copy of $C_n$ corresponding to the $i$-th vertex of $C_m$. Let $\mathbb{C}$ be the set of all copies $C_{n_i}$ in $G$ that are $1$-uniform and $\mathbb{C}'$ be the set of all copies $C_{n_i}$ in $G$ that are not $1$-uniform. The number of mono-indexed edges in in $\mathbb{C}$ is $\frac{mn}{2}$. To find the number of other mono-indexed edges in $G$, we consider the following subcases.

\noindent {\em Subcase-1.1:} Let $n$ is even. Then, each $C_{n_i}$ in  $\mathbb{C}'$ has (at least) $\frac{n}{2}$ mono-indexed vertices and need not have any mono-indexed edge. Since the corresponding vertex $v_i$ in $C_1$ is mono-indexed, the number of mono-indexed edges in the component $C_{n_i}+\{v_i\}$ is $\frac{n}{2}$. Therefore, the number of mono-indexed edges connecting the cycle $C_m$ and the elements of $\mathbb{C}'$ is $\frac{m}{2}.\frac{n}{2}=\frac{mn}{4}$. Hence, the total number of mono-indexed edges in $G$ is $\frac{mn}{2}+\frac{mn}{4}=\frac{3}{4}mn$.

\noindent {\em Subcase-1.2:} Let $n$ is odd. Then, each $C_{n_i}$ in $\mathbb{C}'$ has (at least) $\frac{n+1}{2}$ mono-indexed vertices and has (at least) one mono-indexed edge. Therefore, the total number of mono-indexed edges in $\mathbb{C}'$ is $\frac{m}{2}$. Now, the number of mono-indexed edges in $C_{n_i}+\{v_i\}$ is at least $\frac{(n+1)}{2}$.  Therefore, the number of mono-indexed edge connecting the cycle $C_m$ and the elements of $\mathbb{C}'$ is $\frac{m}{2}.\frac{(n+1)}{2}$. Hence, the total number of mono-indexed edges in $G$ is $\frac{mn}{2}+\frac{m}{2}.\frac{(n+1)}{2}+\frac{m}{2}=\frac{3}{4}m(n+1)$.

\noindent {\em Case-2:} Let $m$ be odd. Then, $C_m$ has $\frac{(m+1)}{2}$ mono-indexed vertices and has $\frac{(m-1)}{2}$ vertices that are not mono-indexed. Also, $C_m$ must have (at least) one mono-indexed edge. Then, by Theorem \ref{T-NMIEG}, $\frac{(m-1)}{2}$ copies of $C_n$ are $1$-uniform. The number of mono-indexed edges in $\mathbb{C}$ is (at least) $\frac{(m-1)n}{2}$. To find the number of other mono-indexed edges in $G$, we consider the following subcases.

\noindent {\em Subcase-1.1:} Let $n$ is even. Then, each $C_{n_i}$ in  $\mathbb{C}'$ has (at least) $\frac{n}{2}$ mono-indexed vertices and need not have any mono-indexed edge. Since the corresponding vertex $v_i$ in $C_m$ is mono-indexed, the number of mono-indexed edges in the component $C_{n_i}+\{v_i\}$ is at least $\frac{n}{2}$. Therefore, the number of mono-indexed edge connecting the cycle $C_1$ and the elements of $\mathbb{C}'$ is at least $\frac{(m+1)}{2}.\frac{n}{2}=\frac{(m+1)n}{4}$. Hence, the total number of mono-indexed edges in $G$ is $1+\frac{(m-1)n}{2}+\frac{(m+1)n}{4}=1+\frac{1}{4}(3m-1)n$.

\noindent {\em Subcase-1.2:} Let $n$ is odd. Then, each $C_{n_i}$ in $\mathbb{C}'$ has (at least) $\frac{n+1}{2}$ mono-indexed vertices and has (at least) one mono-indexed edge. Therefore, the total number of mono-indexed edges in $\mathbb{C}'$ is at least  $\frac{(m+1)}{2}$. Now, the number of mono-indexed edges in $C_{n_i}+\{v_i\}$ is at least  $\frac{(n+1)}{2}$.  Therefore, the number of mono-indexed edge connecting the cycle $C_m$ and the elements of at least $\mathbb{C}'$ is $\frac{(m+1)}{2}.\frac{(n+1)}{2}$. Hence, the total number of mono-indexed edges in $G$ is at least  $1+\frac{(m-1)n}{2}+\frac{(m+1)(n+1)}{4}+\frac{(m+1)}{2}=2+\frac{1}{4}(3m-1)(n+1)$.
\end{proof}

We now proceed to discuss about the sparing numbers of the corona of two graphs, at least one of them being a complete graph. In the following theorem, we estimate the sparing number of the corona of two complete graphs.

\begin{theorem}
The sparing number of the corona $K_m \odot K_n$of two complete graphs $K_m$ and $K_n$ is $\frac{1}{2} [(m-1)(m-2)+mn(n-1)]$.
\end{theorem}
\begin{proof}
The number of edges in $K_m\odot K_n$ is $|E|= \frac{1}{2}m(m-1)+mn+\frac{1}{2}mn(n-1)= \frac{1}{2}m[n^2+n-1+m]$.

By Theorem \ref{T-WKN}, the graph $K_m$ has at most one vertex that is not mono-indexed. Hence, in $K_m\odot K_n$, only one copy of $K_n$, say $K_{n1}$, needs to be $1$-uniform.  Therefore, all the $n$ edges connecting $K_m$ and the copy $K_{n1}$ are not mono-indexed. All other $(m-1)$ copies of $K_n$ can have one vertex having non-singleton set-label. That is, all copies of $K_n$, except $K_{n1}$, have $(n-1)$ edges that are not mono-indexed. Therefore, the total number of edges in $K_m \odot K_n$ that are not mono-indexed , is $(m-1)+n+(m-1)n=mn+m-1$.

Hence, the number of mono-indexed edges in $K_m \odot K_n$ is $\varphi(K_m\odot K_n)=\frac{1}{2}m[n^2+n-1+m]-(mn+m-1)=\frac{1}{2} [(m-1)(m-2)+mn(n-1)]$.
\end{proof}

Next, we consider the corona of a path and a complete graph. We determine the sparing number of $P_m\odot K_n$ in the following theorem.

\begin{theorem}
The sparing number of $P_m\odot K_n$ is $\varphi(P_m\odot K_n)=\frac{1}{2}n(n-1)(m+1)$.
\end{theorem}
\begin{proof}
Consider the following cases.

\noindent {\em Case-1:} If $m$ is even, Then $P_m$ has $\frac{m}{2}$ vertices that are not mono-indexed and $\frac{m+2}{2}$ mono-indexed vertices. Therefore, $\frac{m}{2}$ copies of $K_n$ must be $1$-uniform. For the remaining copies of $K_n$, one vertex can have a mono-indexed set label and hence, there are $(n-1)$ mono-indexed edges between $P_m$ and these copies of $K_n$. Therefore, the total number of mono-indexed edges in $P_m\odot K_n$ is $\frac{m}{2}\frac{n(n-1)}{2}+\frac{m+2}{2}\frac{(n-1)(n-2)}{2}+ \frac{m+2}{2}(n-1)=\frac{1}{2}n(n-1)(m+1)$.

\noindent {\em Case-2:} If $m$ is odd, Then, there are $P_m$ $\frac{m+1}{2}$ mono-indexed vertices and $P_m$ $\frac{m+1}{2}$ vertices that are not mono-indexed. Therefore, $\frac{m+1}{2}$ copies of $K_n$ must be $1$-uniform. For the remaining copies of $K_n$, there are $(n-1)$ mono-indexed edges between $P_m$ and these copies of $K_n$. Therefore, the total number of mono-indexed edges in $P_m\odot K_n$ is $\frac{m+1}{2}\frac{n(n-1)}{2}+\frac{m+1}{2}\frac{(n-1)(n-2)}{2}+ \frac{m+1}{2}(n-1)=\frac{1}{2}n(n-1)(m+1)$. 

This completes the proof.
\end{proof}

The following Theorem estimates the sparing number of $K_n\odot P_m$.

\begin{theorem}
The sparing number of $K_n \odot P_m$ is 
\begin{equation*}
\varphi(K_n\odot P_m)=
\begin{cases}
\frac{1}{2}[(n-1)^2+m(n+1)] & ~~ \text{if $m$ is odd}\\
\frac{1}{2}[(n-1)(n-2)+m(n+1) & ~~ \text{if $m$ is even}.
\end{cases}
\end{equation*}
\end{theorem}
\begin{proof}
By Theorem \ref{T-WKN}, the complete graph $K_n$ has $\frac{1}{2}(n-1)(n-2)$ mono-indexed edges. Since $K_n$ has only one vertex that is not mono-indexed, only one copy of $P_m$ needs to be $1$-uniform and the remaining $(n-1)$ copies of $P_m$ can be labeled alternately by singleton and non-singleton sets. For these copies of $P_m$ we have the following cases.

\noindent {\em Case-1:} If $m$ is odd, then each copy of $P_m$ has $\frac{m+1}{2}$ mono-indexed vertices and hence the total number of mono-indexed edges is $\frac{(n-1)(n-2)}{2}+\frac{(n-1)(m+1)}{2}+m = \frac{1}{2}[(n-1)^2+m(n+1)]$.

\noindent {\em Case-2:} If $m$ is even, then each copy of $P_m$ has $\frac{m}{2}$ mono-indexed vertices and $\frac{m+2}{2}$ vertices that are not mono-indexed.  Therefore, the total number of mono-indexed edges $k_n\odot P_m$ is $\frac{(n-1)(n-2)}{2}+\frac{n-1}{2}m+m = \frac{1}{2}[(n-1)(n-2)+m(n+1)]$.
\end{proof}

The following theorems estimate the sparing number of the coronas of two graphs in which one is a cycle and the other is a complete graph.

\begin{theorem}
The sparing number of $C_m\odot K_n$ is $\varphi(P_m\odot K_n)=\frac{1}{2}mn(n-1)$.
\end{theorem}
\begin{proof}
Consider the following cases.

\noindent {\em Case-1:} If $m$ is odd, Then $C_m$ has $\frac{m-1}{2}$ vertices that are not mono-indexed and $\frac{m+1}{2}$ mono-indexed vertices. Therefore, $\frac{m-1}{2}$ copies of $K_n$ must be $1$-uniform. For the remaining copies of $K_n$, one vertex can have a mono-indexed set label and hence, there are $(n-1)$ mono-indexed edges between $C_m$ and these copies of $K_n$. Therefore, the total number of mono-indexed edges in $C_m\odot K_n$ is $\frac{m-1}{2}\frac{n(n-1)}{2}+\frac{m+1}{2}\frac{(n-1)(n-2)}{2}+ \frac{m+1}{2}(n-1)=\frac{1}{2}mn(n-1)$.

\noindent {\em Case-2:} If $m$ is even, Then, there are $\frac{m}{2}$ mono-indexed vertices and $\frac{m}{2}$ vertices that are not mono-indexed. Therefore, $\frac{m}{2}$ copies of $K_n$ must be $1$-uniform. For the remaining copies of $K_n$, there are $(n-1)$ mono-indexed edges between $C_m$ and these copies of $K_n$. Therefore, the total number of mono-indexed edges in $C_m\odot K_n$ is $\frac{m}{2}\frac{n(n-1)}{2}+\frac{m}{2}\frac{(n-1)(n-2)}{2}+ \frac{m}{2}(n-1)=\frac{1}{2}mn(n-1)$. 
\end{proof}

\begin{theorem}
The sparing number of $K_n \odot C_m$ is 
\begin{equation*}
\varphi(K_n\odot C_m)=
\begin{cases}
\frac{1}{2}[(n-1)^2+m(n+1)] & ~~ \text{if $m$ is odd}\\
\frac{1}{2}[(n-1)(n-2)+m(n+1) & ~~ \text{if $m$ is even}.
\end{cases}
\end{equation*}
\end{theorem}
\begin{proof}
By Theorem \ref{T-WKN}, the complete graph $K_n$ has $\frac{1}{2}(n-1)(n-2)$ mono-indexed edges. Since $K_n$ has only one vertex that is not mono-indexed, only one copy of $C_m$ needs to be $1$-uniform and the remaining $(n-1)$ copies of $C_m$ can be labeled alternately by singleton and non-singleton sets. For these copies of $C_m$ we have the following cases.

\noindent {\em Case-1:} If $m$ is odd, then each copy of $C_m$ has $\frac{m+1}{2}$ mono-indexed vertices and $\frac{m-1}{2}$ vertices that are not mono-indexed. Hence the total number of mono-indexed edges in $K_n\odot C_m$ is $\frac{(n-1)(n-2)}{2}+\frac{(n-1)(m+1)}{2}+m = \frac{1}{2}[(n-1)^2+m(n+1)]$.

\noindent {\em Case-2:} If $m$ is even, then each copy of $P_m$ has $\frac{m}{2}$ mono-indexed vertices and $\frac{m}{2}$ vertices that are not mono-indexed.  Therefore, the total number of mono-indexed edges in $K_n\odot C_m$ is $\frac{(n-1)(n-2)}{2}+\frac{n-1}{2}m+m = \frac{1}{2}[(n-1)(n-2)+m(n+1)]$.
\end{proof}

Next, we study about the sparing number of the corona of two graphs at least one of which is a bipartite graph. In the following theorem, we verify the corona of two complete bipartite graphs. 

\begin{theorem}\label{T-CP-BP}
Let $G_1=K_{m_1,n_1}$ and $G_2=K_{m_2,n_2}$ be two complete bipartite graphs, where $m_i\le n_i$ for $i=1,2$. Then, the sparing number of $G_1\odot G_2$ is $m_2(m_1n_2+n_1)$.
\end{theorem}
\begin{proof}
For $i=1,2$, let $G_i(U_i,V_i,E_i)$ be a bipartite graph with $|U_i|=m_i,~|V_i|=n_i$ and $|E_i=m_in_i$, where $m_i\le n_i$. By Theorem \ref{T-SNBP}, $\varphi(G_i)=0$. 

Label the vertices in $U_1$ of $G_1$ by distinct non-singleton sets and the vertices of $V_1$ of $G_1$ by distinct singleton sets. Label the vertices in $U_2$ of $G_2$ by distinct singleton sets and the vertices of $V_2$ of $G_2$ by distinct non-singleton sets, which are distinct from the set-labels of $G_1$.  

Now, take $|V(G_1)|$ copies of $G_2$ and draw edges between the all vertices of each copy of $G_2$ and the corresponding vertex of $G_1$. Then, all the copies of $G_2$ corresponding to the vertices in $U_1$ of $G_1$ must be $1$-uniform.  That is, $m_1$ copies of $G_2$ are $1$-uniform in $G_1\odot G_2$. Therefore, the total number of mono-indexed edges among these copies of $G_2$ is $m_1m_2n_2$. 

Label the remaining $n_1$ copies of $G_2$ in such a way that the set-labels of the vertices of a copy of $G_2$ are the (suitable) integral multiple of the set-labels of the corresponding set-labels of $G_2$ in such a way that no two vertices in $G_1$ and all copies of $G_2$ have a same set-label. then, the number edges between a vertex of $G_1$ (which are in $V_1$), and the vertices of the corresponding copy of $G_2$ is $m_2$. Therefore, the number of such mono-indexed edges among the vertices in $V_1$ of $G_1$ and the corresponding copies of $G_2$ is $n_1m_2$. 

Therefore, the total number of  mono-indexed edges in $G_1\odot G_2$ is $m_1m_2n_2+n_1m_2 =m_2(m_1n_2+n_1)$.
\end{proof}

Theorem can be generalised for all bipartite graphs as follows and the proof of the theorem is similar to the proof of Theorem \ref{T-CP-BP}. 

\begin{theorem}
Let $G_1(U_1,V_1,E_1)$ and $G_2(U_2,V_2,E_2)$ be two bipartite graphs, where $|U_i|\le |V_i|$ for $i=1,2$. Then, the sparing number of $G_1\odot G_2$ is $|U_1|\,|E_2|+|V_1||U_2| $.
\end{theorem}

Now that we have estimated the sparing number of the corona of two bipartite graphs, we need to find out the sparing number of the corona of a bipartite graph and an odd cycle. 

\begin{theorem}
Let $C_n$ be an odd cycle and $G(V_1,V_2,E)$ be a bipartite graph, where $|V_1|=r$, $|V_2|=s$ and $|E|=q$. Then,
\begin{eqnarray*}
\varphi(C_n\odot G) & = & \frac{1}{2}[(n-1)q+(n+1)r]\\
\text{and}~~~ \varphi(G\odot C_n) & = & \frac{1}{2}[2rn+s(n+1)].
\end{eqnarray*}
\end{theorem}
\begin{proof}
Without loss of generality, let $r\le s$. In $G$, label all the vertices in $V_1$ by distinct singleton sets and label all the vertices in $V_1$ by distinct non-singleton sets. 

\noindent {\em Case-1:} Since $C_n$ is an odd cycle, it has $\frac{n+1}{2}$ mono-indexed vertices and $\frac{n-1}{2}$ vertices that are not mono-indexed. Then, $\frac{n-1}{2}$ copies of $G$, must be $1$-uniform in $C_n\odot G$.  Therefore, the total number of mono-indexed edges in all these copies is $\frac{n-1}{2}q$. Now, label the vertices of the remaining copies of $G$, by suitable integral multiples of the set-labels of corresponding vertices of $G$ in such a way that no two vertices in $C_n\odot G$ have the same set-label. Then, no two vertex in these copies are mono-indexed. But, there are $r$ edges connecting a mono-indexed vertex of $C_n$ and the corresponding copy of $G$. The total number of such mono-indexed edges is $\frac{n+1}{2}r$. Therefore, the sparing number of $C_n\odot G$ is $\frac{1}{2}[(n-1)q+(n+1)r]$.

\noindent {\em Case-2:} Since $r\le S$, label all the vertices in $V_1$ of $G$ by distinct non-singleton sets and label all the vertices in $V_2$ by distinct singleton sets. Then, $r$ copies of $C_n$ must be $1$-uniform in $G\odot C_n$. The number mono-indexed copies all together is $rn$. Now, the remaining copies of $C_n$, corresponding to the vertices in $V_2$ of $G$, contain $\frac{n+1}{2}$ mono-indexed vertices each and hence there are $\frac{n+1}{2}$ mono-indexed edges between a vertex in $V_2$ and its corresponding copy of $C_n$ is $\frac{n+1}{2}$. Therefore, the number of such mono-indexed edges is $\frac{n+1}{2}s$. Hence, the sparing number of $G\odot C_n$ is $rn+\frac{n+1}{2}s=\frac{1}{2}[2rn+(n+1)s]$.
\end{proof}

It remains to find the sparing number of the corona of two graphs, one of which is a complete graph and the other is a bipartite graph. Then, we have

\begin{theorem}
Let $K_n$ be an odd cycle and $G(V_1,V_2,E)$ be a bipartite graph, where $|V_1|=r$, $|V_2|=s$ and $|E|=q$. Then,
\begin{eqnarray*}
\varphi(K_n\odot G) & = & \frac{1}{2}[(n-1)(n-2)+2q+2r(n-1)]\\
\text{and}~~~ \varphi(G\odot K_n) & = & \frac{1}{2}(n-1)[rn+s(n-2)].
\end{eqnarray*}
\end{theorem}
\begin{proof}
Let $r\le s$. In $G$, label all the vertices in $V_1$ by distinct singleton sets and label all the vertices in $V_1$ by distinct non-singleton sets. In $G$, label all the vertices in $V_1$ by distinct singleton sets and label all the vertices in $V_1$ by distinct non-singleton sets. 

\noindent {\em Case-1:} in $K_n$, every vertex is adjacent to all other vertices of it, only one vertex in $K_n$ can have a non-singleton set-label. Therefore, one copy of $G$ must be $1$-uniform in $K_n\odot G$. The vertices of the remaining copies of $G$ shall be labeled by the integral multiples of the set-labels of the corresponding vertices of $G$. Then, the total number of mono-indexed edges in $K_n\odot G$ is $q+(n-1)r+\frac{1}{2}(n-1)(n-2)=\frac{1}{2}[(n-1)(n-2)+2q+2r(n-1)]$.

\noindent {\em Case-2:} In $G\odot K_n$, $r$ copies of $K_n$ must be $1$-uniform. The remaining copies of $K_n$ can have $(n-1)$ vertices that are mono-indexed. Therefore, the total number of mono-indexed edges is $r\frac{1}{2}n(n-1)+s\frac{1}{2}(n-1)(n-2)=\frac{1}{2}(n-1)[rn+s(n-2)]$.

\end{proof}

\section{Conclusion}

In this paper, we have discussed about the sparing number of the corona of weak IASI graphs.  Some problems in this area are still open. Uncertainty in the adjacency pattern of different graphs makes this study complex. An investigation to verify the admissibility of weak IASIs by other graph products of two arbitrary graphs and to determine the corresponding sparing numbers seems to be fruitful.

\end{document}